\RequirePackage{fix-cm}
\documentclass[smallextended]{svjour3}       
\smartqed  
\usepackage{graphicx}
\usepackage{enumerate}
\usepackage{mathbbol}
\usepackage{amssymb}
\usepackage{braket}
\usepackage{longtable}
\usepackage[colorlinks]{hyperref}
\usepackage{caption}
\usepackage{mathtools}
\usepackage{microtype}
\newcommand{\Z}{\mathbb{Z}}

\newcommand{\bs}{\backslash}

\renewcommand\AA{{\mathcal A}}

\newcommand\LL{{\mathcal L}}

\newcommand\acl{\hbox{\rm acl}}

\newcommand\<{\langle}
\renewcommand\>{\rangle}

\newcommand{\tp}{\mathrm{tp}}

\newcommand\abar{\bar{a}}

\newcommand\bbar{\bar{b}}
\newcommand\cbar{\bar{c}}
\newcommand\dbar{\bar{d}}
\newcommand\ebar{\bar{e}}

\newcommand\mbar{\bar{m}}

\newcommand\ubar{\bar{u}}
\newcommand\vbar{\bar{v}}
\newcommand\wbar{\bar{w}}
\newcommand\xbar{\bar{x}}
\newcommand\ybar{\bar{y}}
\newcommand\zbar{\bar{z}}
\newcommand\A{{\mathcal A}}

\newcommand\E{{\mathcal E}}

\newcommand\myrestriction{\mathord\restriction}
\newcommand{\mr}[1]{\myrestriction_{#1}}
\def\stp{{\rm stp}}
\def\tp{{\rm tp}}
 
 \newtheorem{fact}[lemma]{Fact}
 \newenvironment{claimproof}[1][Proof of Claim]{\begin{proof}[#1]}{\end{proof}}
 
\begin{document}

\title{Mutual algebraicity and cellularity\thanks{Second author partially supported
by NSF grant DMS-1855789}
}


\author{Samuel Braunfeld        \and
        Michael C. Laskowski 
}


\institute{S. Braunfeld \at
              4176 Campus Dr, College Park, MD 20742 \\
              \email{sbraunf@umd.edu}           
           \and
           M. Laskowski \at
              4176 Campus Dr, College Park, MD 20742 \\
              \email{laskow@umd.edu}  
}

\date{Received: date / Accepted: date}

\maketitle

\begin{abstract}
We prove two results intended to streamline proofs about cellularity that pass through mutual algebraicity. First, we show that a countable structure $M$ is cellular if and only if $M$ is $\omega$-categorical and mutually algebraic. Second, if a countable structure $M$ in a finite relational language is mutually algebraic non-cellular, we show it admits an elementary extension adding infinitely many infinite MA-connected components. 

Towards these results, we introduce MA-presentations of a mutually algebraic structure, in which every atomic formula is mutually algebraic. This allows for an improved quantifier elimination and a decomposition of the structure into independent pieces. We also show this decomposition is largely independent of the MA-presentation chosen.
\keywords{model theory \and omega-categoricity \and stable theories \and hereditary classes}
 \subclass{03C15 \and 03C45}
\end{abstract}

\section{Introduction}

 Cellular structures are a particularly simple class of countable structures that appear as a dividing line in many combinatorial problems concerning countable structures or hereditary classes of finite structures. Cellularity (Definition~\ref{def:cell}) plays a  leading role when counting the number of countable structures of a given age \cite{MPW}, counting the finite models of a given size in a hereditary class \cite{LT2}, or counting the number of substructures of a countable structure up to isomorphism \cite{LM}. It is usually easy to check that cellular structures are well-behaved in these problems, so the work consists of showing that the non-cellular case exhibits the appropriate ``many models'' behavior. The easiest strategy is seemingly to pass through the weaker notion of mutual algebraicity (Definition~\ref{def:madef}) splitting the problem into the non-mutually algebraic case and the mutually algebraic but non-cellular case. This strategy is explicitly employed in the companion paper \cite{BLsib} to count the number of structures bi-embeddable with a given countable structure.  In this paper, we provide results describing the gap between mutual algebraicity and cellularity, for use in the second of the cases above.
 
 
 In our strategy for problems concerning cellularity, the first step is to produce complicated behavior in the non-mutually algebraic case, likely using the Ryll-Nardzewski-type characterization of mutual algebraicity in a finite relational language from \cite{LT1}. Our first result shows that in the $\omega$-categorical setting, this is already enough.
 
\begin{theorem}[Theorem \ref{thm:arbitrary}]
 Suppose $M$ is a countable structure in a countable language.  Then $M$ is cellular if and only if $M$ is mutually algebraic and $\omega$-categorical.
\end{theorem}

The next step of the strategy is to show that if $M$ is mutually algebraic but not cellular, then models of its theory are still sufficiently complex.  This is the main part of the paper.
At first blush, it is not clear that cellularity and mutual algebraicity are related.  By definition, a cellular structure $M$ admits a natural partition into finite pieces.
Depending on the language a mutually algebraic structure $M$ is presented in, such a partition might not be readily describable.  
 As examples, if we take $T$ to be the theory of $(\Z,succ)$, the theory of the integers with a binary successor relation, then any model $M$ of $T$ is mutually algebraic, and $M$ is naturally partitioned  into its `$\Z$-classes'.  The salient feature is that the binary successor relation is mutually algebraic, i.e. essentially bounded degree.
By contrast, the arguably simpler theory of $Th(M,E)$, where $E$ is an equivalence relation on $M$ with two classes, both infinite, is mutually algebraic (in fact, cellular), yet the atomic formula $E(x,y)$ is not mutually algebraic.  In this case, the cellular partition of $M$ is into singletons, with singletons from the two classes distinguished, which requires naming a constant.
In Section~2, we introduce the notion of an {\em MA-presentation}, in which all atomic relations are mutually algebraic.  We show that MA-presented structures $M$ admit a stronger quantifier elimination result and there is a natural equivalence relation $\sim$ that partitions $M$ as in the definition of cellularity.  
We also define a notion of two structures with the same universe but in different languages to be {\em associated} and we prove that every mutually algebraic structure is associated to an MA-presented structure.  Although there may be several MA-presented structures associated to a given mutually algebraic $M$, the corresponding partitions into $\sim$-classes
are largely the same.  

In Section 3, we use the fact that any mutually algebraic $M$ has many associated MA-presented structures to give several characterizations of cellularity, e.g., 
 a mutually algebraic $M$ is cellular if and only if there is a uniform finite bound on the size of $\sim$-classes in some/every associated
MA-presented $M'$.
As a byproduct of this analysis, we prove the following result, which is used by the authors in \cite{BLsib} to obtain ``many models'' behavior for mutually algebraic non-cellular structures, since it shows they will in some sense encode a partition with infinitely many infinite classes in some elementary extension.
 
 \begin{theorem}[Theorem \ref{thm:ma noncell}]
 Let $\LL$ be finite relational, and suppose $M$ is mutually algebraic but non-cellular countable $\LL$-structure. Then there is some $M^* \succ M$ such that $M^*$ contains infinitely many new infinite MA-connected components, pairwise isomorphic over $M$. Furthermore, we may take the universe of $M^*$ to be the universe of $M$ together with these new components.
 \end{theorem}
 
Also in Section 3,  we  obtain other characterizations of cellularity, including the following, which does not mention partitions.
 
  \begin{theorem}[Theorem \ref{thm:mon cat}]
  A countable structure $M$ is monadically $\omega$-categorical if and only if $M$ is cellular.
  \end{theorem}
  
  The paper is largely self-contained.  In some places, there are references made to $\omega$-stable and superstable theories, but most of what is used about them is elementary and can be found in e.g., \cite{TZ}.
 
\section{Mutual algebraicity and MA-presentations}

Throughout this paper, we only consider languages $\LL$ with no function symbols (i.e., $\LL$ only has relation and constant symbols).  As all the properties we consider are preserved by passing between a function and its graph, this is not a significant restriction.

\begin{definition}  \label{def:ONE}
Given a structure $M$, an $n$-ary relation $R(\xbar)$ is {\em mutually algebraic} if there is a constant $K$ such that for each $m \in M$, the number of tuples $\bbar \in M^n$ such that $M \models R(\bbar)$ and $m \in \bbar$ is at most $K$.
\end{definition}

Note that every unary relation is mutually algebraic.

\begin{definition}  \label{def:madef}
Given an $\LL$-structure $M$, let $\LL_M$ be $\LL$ expanded by constant symbols for every element of $M$.
 $M$ is {\em mutually algebraic} if every $\LL_M$-formula is equivalent to a boolean combination of mutually algebraic $\LL_M$ formulas.
\end{definition}

\begin{theorem} [\cite{MA} {Theorem 3.3}, \cite{MA0} {Proposition 4.1, Theorem 4.2}] \label{thm:ma char}
Given an $\LL$-structure $M$, let $\LL_M$ be $\LL$ expanded by constant symbols for every element of $M$. The following are equivalent.
\begin{enumerate}
\item $M$ is mutually algebraic.
\item Every atomic $\LL$-formula is equivalent to a boolean combination of quantifier-free mutually algebraic $\LL_M$-formulas.
\item Every $\LL_M$-formula is equivalent to a boolean combination of $\LL_M$-formulas of the form $\exists \ybar \theta(\xbar, \ybar)$, where $\theta(\xbar, \ybar)$ is quantifier-free mutually algebraic.
\item $M$ is weakly minimal and trivial.
\end{enumerate}

\end{theorem}

While Theorem~\ref{thm:ma char} is useful, the dependence on $\LL_M$--formulas can be awkward in applications.  
To alleviate this, we introduce a stronger notion of an MA-presented
structure and with Proposition~\ref{prop: qe}, we see that such a structure admits a better reduction of quantifier complexity.   
However, as we will see in Lemma~\ref{lem:basic}, every mutually algebraic structure $M$ is {\em associated} (see Definition~\ref{def:associated}) to an MA-presented structure $M'$ in a new language.

\subsection{MA-presented structures}

\begin{definition} An $\LL$-structure $M$ is {\em MA-presented} if every atomic $\LL$-formula is mutually algebraic.
$M$ is {\em finitely MA-presented} if, in addition, $\LL$ is finite relational.
\end{definition}

We begin by bounding the quantifier complexity of formulas in an MA-presented structure.

\begin{definition}
Let $\E=\{\exists\zbar\theta(\ybar,\zbar),$ where $\theta(\ybar,\zbar)$ is quantifier-free and mutually algebraic $\LL$-formula$\}$.  We allow $\lg(\zbar)=0$, so as $M$ is MA-presented,
every atomic $\LL$-formula is in $\E$.

Let $\E^*:=\{$all $\LL$-formulas equivalent to boolean combinations and adjunctions of free variables of formulas from $\E\}$. 
As every $\gamma(\xbar)\in\E$ is mutually algebraic, the set $\E$ of formulas is {\bf not} closed under adjunction of dummy variables, although $\E^*$ is.
\end{definition}

We note a dichotomy among partitioned formulas $\beta(x,\ybar):=\exists\zbar\theta(x,\ybar,\zbar)\in\E$ with $\lg(x)=1$:
On one hand, if both $\lg(\ybar)=\lg(\zbar)=0$, i.e., if $\beta(x):=\theta(x)$ is quantifier-free and mutually algebraic, then $\neg\beta(x)$ is also quantifier-free, mutually algebraic
(recall that every formula in one free variable is mutually algebraic) hence $\neg\beta\in\E$.  On the other hand, if at least one of $\lg(\ybar),\lg(\zbar)>0$, then for some $K\in\omega$,
$M\models\forall\ybar\exists^{\le K}x\beta(x,\ybar)$, i.e., $\neg\beta(M,\cbar)$ is cofinite for all $\cbar\in M^{\lg(\ybar)}$.

\begin{proposition} \label{prop: qe}  Suppose $M$ is an infinite, MA-presented $\LL$-structure.  Then every $\LL$-formula is equivalent to a formula in $\E^*$.
\end{proposition}

\begin{proof}
This proof is similar to the proof of Theorem~4.2 of \cite{MA0}.  
 Via Disjunctive Normal Form and the dichotomy above, every formula in $\E^*$ in which the variable $x$ appears in each element of $\E$ can be written as
$$\bigvee_k\left[\bigwedge_i \alpha_{i,k}(x,\ybar_{i,k})\wedge\bigwedge_j \neg\beta_{j,k}(x,\ybar_{j,k})\right]$$
where $\neg\beta_{j,k}(M,\cbar)$ is cofinite for all $j,k,$ and $\cbar$.

Note that $\E^*$ contains all atomic $\LL$-formulas and is closed under boolean combinations.  Thus, to prove the Proposition, it suffices to show that $\E^*$ is closed under projections.
By the standard form given above and the fact that $\exists$ commutes with $\bigvee$, it suffices to show that
$$\delta(\ybar):=\exists x \left(\bigwedge_{i}\alpha_i(x,\ybar_i)\wedge\bigwedge_{j=1}^J\neg\beta_j(x,\ybar_j)\right)\in\E^*$$
where each $\alpha_i,\beta_j\in\E$ and where $\neg\beta_j(M,\cbar_j)$ is cofinite for all $j$ and $\cbar_j$ from $M$.
We split into cases.

\medskip\noindent{\bf Case 1.}  There is no $\alpha_i$.
\medskip

In this case, as $M$ is infinite and each $\neg\beta_j(M,\cbar_j)$ is cofinite, $\delta(\ybar):=\exists x \bigwedge_j \neg\beta_j(x,\ybar_j)$ holds for all $\ybar$.

\medskip\noindent{\bf Case 2.}  There is at least one $\alpha_i$.  
\medskip

As the conjunction of elements of $\E$ with a common variable symbol is also in $\E$, we may assume there is exactly one $\alpha^*(x,\ybar^*)$.  There are now two subcases.

\medskip\noindent{\bf Subcase 2a).}  $\ybar^*=\emptyset$ and $\alpha^*(M)$ is infinite.
\medskip

Then as in Case 1, $\delta(\ybar):=\exists x [\alpha^*(x)\wedge\bigwedge_j \neg\beta_j(x,\ybar_j)]$ holds for all $\ybar$.

\medskip\noindent{\bf Subcase 2b).}  Not Subcase 2a).
\medskip

In this case, for some $r^*\in\omega$, $M\models\forall\ybar^*\exists^{\le {r^*}}x\alpha^*(x,\ybar^*)$.
For each $j$, put $\gamma(x,\ybar_j'):=\alpha^*(x,\ybar^*)\wedge\beta_j(x,\ybar_j)$.  Thus,  $\gamma_j(M,\cbar_j')\subseteq\alpha^*(M,\cbar^*)$ for all $j$ and all $\cbar\in M^{\lg(\ybar)}$.
Also, as $\alpha^*$ and $\beta_j$ contain a common free variable, namely $x$, each $\gamma_j(x,\ybar_j')$ is mutually algebraic.
It follows that $\delta(\ybar)$ is equivalent to
$$\bigvee_{1\le r\le r^*}\left(\exists^{=r}x\alpha^*(x,\ybar^*)\wedge \exists^{<r}x \bigvee_{j=1}^J\gamma_j(x,\ybar_j')\right)$$
so it suffices to show that this formula is in $\E^*$.  

As in the proof of Lemma~2.4 of [8], for each $r$ and for each $S\subseteq\{1,\dots,J\}$,
$\exists^{\le r}x\bigwedge_{j\in S}\gamma_j(x,\ybar_j')\in\E$.  From this, it follows easily that $\exists^{=r}x\bigwedge_{j\in S}\gamma_j(x,\ybar_j')$ and 
$\exists^{<r}x\bigwedge_{j\in S}\gamma_j(x,\ybar_j')$ are in $\E^*$.   Finally, by the inclusion-exclusion principle, it follows that $\bigvee_{j=1}^J\gamma_j(x,\ybar_j')$ is also in $\E^*$
and we finish.
\end{proof}

This restriction on quantifier complexity will be used in the analysis of the following natural decomposition of an MA-presented structure into independent components, analogous to the decomposition of a graph into connected components. In fact, the mate relation defined next is the same as the edge relation of the Gaifman graph, and $\sim$ the same as its connectedness relation. 

\begin{definition}  Suppose $M$ is an MA-presented $\LL$-structure.  For $a,b\in M$, call $b$ a {\em mate of $a$} if there is some $\cbar$ from $M$ and the type of
(some permutation of) $ab\cbar$ contains an atomic 
$\LL$-formula.  Being a mate is reflexive and symmetric.  Let $\sim$ be the transitive closure of the mate relation.  As $\sim$ is an equivalence relation, write
$[a]_\sim$ for the $\sim$-class of $a\in M$, and call $\sim$-classes {\em MA-connected components}.
\end{definition} 

Observe that if $M$ is an MA-presented $\LL$-structure, then $M\models\neg R(\dbar)$ for all $\LL$-atomic $R(\xbar)$ and all $\dbar\in M^{\lg(\xbar)}$ whenever $\dbar$ is not
contained in a single $\sim$-class.  

\begin{definition}  Suppose $M$ is an MA-presented $\LL$-structure.  A {\em component map} is a bijection $f:M\rightarrow M$ such that $f([a]_\sim)=[f(a)]_\sim$ (setwise) for all $a\in M$.
An {\em $\LL$-component map} is a component map such that for each $a\in M$, $f\mr{[a]_\sim}:[a]_\sim\rightarrow [f(a)]_\sim$ is an $\LL$-isomorphism.
\end{definition}

The following Lemma, similar to \cite[{Proposition 4.4}]{MA} although working with a single structure, follows immediately from the observation above.

\begin{lemma}  \label{lem: component} If $M$ is an MA-presented $\LL$-structure, then $f \colon M \to M$ is an automorphism if and only if it is an $\LL$-component map.
\end{lemma}

The following definition gives a syntactic characterization of the $\sim$-relation.

\begin{definition}  A {\em linked $\LL$-atomic conjunction} is an $\LL$-formula of the form $\theta(\wbar):=\bigwedge_{i<n} \alpha_i(\wbar_i)$, where each $\alpha_i(\wbar_i)$ is 
$\LL$-atomic and $\wbar_i\cap\wbar_{i+1}\neq\emptyset$ for all $i<n-1$.  (If $n=1$, this last condition is vacuous, hence every $\LL$-atomic formula is a linked $\LL$-atomic conjunction.)
\end{definition}

Note that if every atomic $\LL$-formula is mutually algebraic (e.g., when $M$ is MA-presented), then as $\phi(\ubar)\wedge\psi(\vbar)$ is mutually algebraic whenever
$\phi(\ubar),\psi(\vbar)$ are mutually algebraic and $\ubar\cap\vbar\neq\emptyset$, it follows every linked $\LL$-atomic conjunction $\theta(\wbar)$ is mutually algebraic as well.

\begin{lemma}  \label{lem:linked}  Suppose $M$ is an infinite, MA-presented $\LL$-structure.  For $a,b\in M$, $a\sim b$ if and only if $\tp(a,b)$ contains $\exists\ubar \theta(x,y,\ubar)$
for some  linked, $\LL$-atomic conjunction
$\theta(x,y,\ubar)$.
\end{lemma}

\begin{proof}  First, assume $a\sim b$.  Choose $\set{c_0,\dots,c_n}\subseteq M$ such that $c_0=a$, $c_n=b$, and $c_{i+1}$ is a mate of $c_i$ for each $i<n$.
For each $i<n$, choose an $\LL$-atomic $\alpha_i(x_i,x_{i+1},\zbar_i)$ witnessing $c_{i+1}$ is a mate of $c_i$.  Without loss, assume $\zbar_i$ and $\zbar_j$ are disjoint for distinct $i,j$.  Then, writing $\wbar_i$ for $x_ix_{i+1}\zbar_i$,  $\theta(\wbar):=\bigwedge_{i<n}\alpha_i(\wbar_i)$ is a linked $\LL$-atomic conjunction and $M\models\exists\ubar \theta(a,b,\ubar)$, where $\ubar:=\wbar\setminus\set{x_0,x_n}$.

Conversely, assume $M\models\exists\ubar\theta(a,b,\ubar)$, where $\theta(\wbar):=\bigwedge_{i<n} \alpha_i(\wbar_i)$, with each $\alpha_i$ $\LL$-atomic and $\wbar_i\cap\wbar_{i+1}\neq\emptyset$.  For each $i<n$, choose a variable symbol $x_i\in\wbar_i\cap\wbar_{i+1}$.  Choose $\cbar\in M^{\lg(\ubar)}$ such that $M\models\theta(a,b,\cbar)$, let $\cbar_i$ denote the restriction of $\cbar$ corresponding to $\wbar_i$, and let $d_i$ be the singleton in $\cbar_i$ corresponding to $x_i$.   Without loss, assume $a\in\cbar_0$ and $b\in\cbar_{n-1}$.  Then $a, d_0$ and $b,d_{n-1}$ are mates, as are $d_i,d_{i+1}$ for each $i<n-1$.  Thus, $a\sim b$.
\end{proof}

For elements $a,b\in\acl(\emptyset)$, $a\sim b$ can be rather strange (e.g., if $c,d$ are constant symbols then the formula $\phi(x,y):=(x=c\wedge y=d)$ is mutually algebraic).
However, outside of $\acl(\emptyset)$, the $\sim$ relation characterizes dependence.
The following definition appears in \cite{LT1}.

\begin{definition}  Let $M$ be any $\LL$-structure.  A mutually algebraic $\LL$-formula $\theta(\zbar)$ {\em supports an infinite array} if there is an infinite set $\set{ \dbar_i:i\in\omega}$ 
such that $M\models\theta(\dbar_i)$ for each $i$ and $\dbar_i\cap\dbar_j=\emptyset$ for all $i\neq j$.
\end{definition}

\begin{lemma}  Suppose $M$ is any structure and $\theta(\zbar)$ is a mutually algebraic formula.  
Then $\theta(\zbar)$ supports an infinite array if and only if $\theta(\zbar)$ is not algebraic, i.e., $M\models\exists^\infty\zbar\theta(\zbar)$.
\end{lemma}

\begin{proof}  Left to right is obvious, so assume $\theta(\zbar)$ is not algebraic.  Choose any infinite set $\{\dbar_i:i\in\omega\}$ of distinct realizations of $\theta(\zbar)$.
By the finite $\Delta$-system lemma, there is a finite set $R$ and an infinite $I\subseteq\omega$ such that $\dbar_i\cap\dbar_j=R$ for all distinct $i,j\in I$.
As $I$ is infinite and $\theta(\zbar)$ is mutually algebraic, we must have $R=\emptyset$, hence $\theta(\zbar)$ supports an infinite array.
\end{proof}

\begin{lemma}  \label{lem:support}  Suppose $M$ is MA-presented, $\cbar\in M^k$, and $\cbar\cap\acl(\emptyset)=\emptyset$.
\begin{enumerate}
\item  If  $\tp(\cbar)$ contains a quantifier-free, mutually algebraic
$\LL$-formula $\theta(\zbar)$ with $\lg(\zbar)=\lg(\cbar)$, then
 $\tp(\cbar)$ contains a linked $\LL$-atomic conjunction $\rho(\zbar)$.
 \item  If $\tp(\cbar)$ contains a mutually algebraic $\LL$-formula $\delta(\xbar)$ with $\lg(\xbar)=\lg(\cbar)$, then there are $\zbar\supseteq\xbar$, $\dbar\supseteq\cbar$ with
 $\lg(\dbar)=\lg(\zbar)$, and a linked $\LL$-atomic conjunction $\rho(\zbar)\in\tp(\dbar)$.
 \item If $\cbar$ satisfies a quantifier-free mutually algebraic $\LL_M$-formula $\theta(\xbar,\mbar)$ with $\cbar\cap\acl(\mbar)=\emptyset$. Then there are
 $\zbar\supseteq\xbar$, $\dbar\supseteq\cbar$ with
 $\lg(\dbar)=\lg(\zbar)$, and a linked $\LL$-atomic conjunction $\delta(\zbar)\in\tp(\dbar)$.
 \end{enumerate}
\end{lemma}

\begin{proof}  By passing to a large enough elementary extension, we may assume $M$ is $|\LL|^+$-saturated.  
For (1), by writing $\theta(\zbar)$ in Disjunctive Normal Form and choosing a disjunct in $\tp(\cbar)$, we may assume $\theta(\zbar)$ has the form
$$\bigwedge_{i\in S} \alpha_i(\zbar_i)\wedge\bigwedge_{j\in U} \neg\beta_j(\zbar_j)$$
where each $\alpha_i,\beta_j$ are $\LL$-atomic.  Now, since $M$ is MA-presented, each $\alpha_i(\zbar_i),\beta_j(\zbar_j)$ is mutually algebraic, and since
$\cbar\cap\acl(\emptyset)=\emptyset$, the saturation of $M$ implies that $\theta(\zbar)$ supports an infinite array.  Thus, by Lemma~A.1 of \cite{LT1}, 
there is an $S_0\subseteq S$ such that $\rho(\zbar):=\bigwedge_{i\in S_0}\alpha_i(\zbar_i)$ is a linked, $\LL$-atomic conjunction with $\bigcup\{\zbar_i:i\in S_0\}=\zbar$.
As $\theta(\zbar)\vdash\rho(\zbar)$, $\rho(\zbar)\in\tp(\cbar)$.

For (2), by Proposition~\ref{prop: qe}, we may assume $\delta(\xbar)$ has the form there.  By choosing a disjunct satisfied by $\cbar$, we may assume
$\delta(\xbar)$ has the form
$$\bigwedge_{i\in S}\phi_{i}(\xbar_i)\wedge\bigwedge_{j\in U} \neg\psi_j(\xbar_j)$$
where each $\phi_i(\xbar_i),\psi_j(\xbar_j)$ is from $\E$, and so is mutually algebraic. As
$\delta(\xbar)$ supports an infinite array, again we have a subset $S_0\subseteq S$ with $\gamma(\xbar):=\bigwedge_{i\in S_0} \phi_i(\xbar_i)$, and $\{\xbar_i:i\in S_0\}$ linked with
$\bigcup\{\xbar_i:i\in S_0\}=\xbar$.    Now, each $\phi_i(\xbar_i)$ has the form $\exists\wbar_i\eta_i(\xbar_i,\wbar_i)$ where $\eta_i(\xbar_i,\wbar_i)$ is quantifier-free and mutually algebraic.
Thus, $\gamma(\xbar)$ is equivalent to a formula $\exists \wbar^*\theta(\xbar,\wbar^*)$ with $\theta(\xbar,\wbar^*)$  quantifier free and mutually algebraic since $\{\eta_i(\xbar_i,\wbar_i):i\in S_0\}$ are linked.  Additionally, $M\models\gamma(\cbar)$.  Write $\theta$ as $\theta(\zbar)$ with $\zbar\supseteq\xbar$.  Choose $\dbar\supseteq \cbar$ such that $M\models\theta(\dbar)$.
As $\theta(\zbar)$ is mutually algebraic, $\dbar\cap\acl(\emptyset)=\emptyset$.  Thus, we can apply (1) to $\theta(\zbar)$ and $\dbar$, getting $\rho(\zbar)$ as required.

(3) 
Let $M_{\mbar}$ be the expansion of $M$ naming the constants of $\mbar$. Applying (1) to $\theta(x, \mbar)$ in $M_{\mbar}$,
$tp(\cbar)$ contains a linked, 
$\LL_{\mbar}$-atomic conjunction
 $\rho(\xbar,\mbar):=\bigwedge_{i\in S_0}\alpha_i(\xbar_i,\mbar_i)$ with $\bigcup\{\xbar_i:i\in S_0\}=\xbar$.
As $M$ is MA-presented, every atomic formula $\alpha_i(\xbar,\wbar_i)$ is mutually algebraic, hence by choosing disjoint variables for each $\wbar_i$ and letting
$\wbar^\#=\<\wbar_i:i\in S_0\>$,
we have $M\models\exists \wbar^\#\delta(\cbar,\wbar^\#)$
where $\delta(\cbar,\wbar^\#):=\bigwedge_{i\in S_0} \alpha_i(\xbar_i,\wbar_i)$.  Choose any $\dbar\supseteq\cbar$ with $M\models\delta(\dbar)$ and we finish, taking
$\zbar:=\xbar$\^{}$\wbar^\#$.
\end{proof}

\begin{proposition}  \label{prop: acl}  Suppose $M$ is an $MA$-presented $\LL$-structure and let $a\in M\setminus\acl(\emptyset)$.  Then 
$[a]_\sim=\acl(a)\setminus\acl(\emptyset)$.
\end{proposition}

\begin{proof}
Assume $M$ is MA-presented and fix $a\in M\setminus\acl(\emptyset)$.
First, assume $b\sim a$.  By Lemma~\ref{lem:linked} choose a linked $\LL$-atomic conjunction $\theta(x,y,\ubar)$ and $\cbar$ such that
$M\models\theta(a,b,\cbar)$.  As $\theta$ is mutually algebraic, it follows that $b\in\acl(a)$ and $a\in\acl(b)$.  Thus, if $b\in\acl(\emptyset)$, we would
have $a\in\acl(\emptyset)$ as well, which it is not.  So $b\in \acl(a)\setminus\acl(\emptyset)$.

Conversely, assume $b\in \acl(a)\setminus\acl(\emptyset)$.  As $M$ is mutually algebraic, $Th(M)$ is weakly minimal by Theorem~\ref{thm:ma char},
and so $\acl$ satisfies exchange.  That is, since $b\in \acl(a)\setminus\acl(\emptyset)$, we also have $a\in\acl(b)\setminus\acl(\emptyset)$.
Thus, there is an integer $K$ and a pair of $\LL$-formulas $\phi(x,y),\psi(x,y)\in\tp(a,b)$ witnessing the algebraicity, i.e.,
$M\models\exists^{\le K}x\phi(x,b)$ and $M\models\exists^{\le K}y\psi(a,y)$
Let $$\delta(x,y):=\phi(x,y)\wedge\psi(x,y)\wedge \exists^{\le K}u\phi(u,y)\wedge\exists^{\le K}v\psi(x,v)$$
Then $\delta(x,y)\in\tp(a,b)$ and is mutually algebraic directly from Definition~\ref{def:ONE}.
Since  $\{a,b\}\cap\acl(\emptyset)=\emptyset$,  $b\sim a$  follows from Lemma~\ref{lem:support}(2) and Lemma~\ref{lem:linked}.
\end{proof}

Much of the development in this subsection generalizes Section 4 of \cite{MA}, which studies elementary extensions of a mutually algebraic model $M$.
There, for $N\succ M$ and $a\in (N-M)$, the component $[a]$ of $a$ was defined as $\acl(a)\bs M$, and every such $N$ is a union of components over $M$. This agrees with our notation, as the natural expansion of $N$ to an $\LL_M$ structure is MA-presented with $\acl(\emptyset)=M$, but importantly, here we name fewer constants.

\subsection{Associating a mutually algebraic structure with an MA-presented structure}

In this section, we define what we mean by an {\em association} of structures in different languages and prove that every mutually algebraic structure is
associated to an MA-presented structure.  This procedure is far from unique, but we explore how different the $\sim$-relations can be.

\begin{definition}   Suppose $\LL\subseteq \LL^*$ are languages and $M$ is an $\LL$-structure.  An $\LL^*$-structure $M^*$ is a {q.f.\ expansion by definitions of $M$}
if $M^*$ is an expansion of $M$ (i.e., $M^*$ and $M$ have the same universes and the interpretations of every $\LL_M$-formula are the same in both structures) such that
every atomic $\LL^*$-formula is equivalent to a quantifier-free $\LL_M$-formula.
\end{definition}

The following facts are easy.

\begin{fact} \label{fact: expansion} Suppose $\LL\subseteq \LL^*$ and $M^*$ is a q.f.\ expansion of $M$ by definitions.  Then:
\begin{enumerate}
\item  Every q.f.\ $\LL^*_M$-formula is equivalent to a q.f. $\LL_M$-formula.
\item  Every $\LL^*_M$-formula $\theta^*(\xbar)$ is equivalent to an $\LL_M$-formula $\gamma(\xbar)$.
\item  For every $N\succeq M$ there is a unique $N^*\succeq M^*$ such that $N^*$ is a q.f.\  expansion of $N$ by definitions; and
\item  For every $N^*\succeq M^*$, the $\LL$-reduct $N\succeq M$ and $N^*$ is a q.f.\ expansion by definitions of $N$.
\item  If $\LL^*$ (and hence $\LL$) are finite, relational, then there is some finite $\mbar$ from $M$ so that every $\LL^*$-formula $\phi(\xbar)$ is equivalent  in $M^*$ to
some $\LL_{\mbar}$-formula $\psi(\xbar,\mbar)$.
\end{enumerate}
\end{fact}

\begin{proof} (1)  By taking boolean combinations, one immediately gets that every q.f.\ $\LL^*$-formula is equivalent to a q.f.\ $\LL_M$-formula.
Then (1) follows from this by taking specializations.

(2)
As $M$ and $M^*$ have the same universes, one proves by induction on the complexity of $\LL^*$-formulas that every $\LL^*$-formula is equivalent to
an $\LL_M$-formula.  By taking specializations, (2) follows.

(3)  Given any atomic $\LL^*$-formula $\alpha(\xbar)$, let $\theta(\xbar)$ be a q.f.\ $\LL$-formula equivalent to it.  Now interpret $\alpha(N^*)$ as $\theta(N)$.
This works, and is the unique way of obtaining such an $N^*$.

(4)  Easy.

(5)  As $\LL^*$ is finite relational, let $\{\alpha_i(\xbar_i):i<K\}$ enumerate the (finitely many) $\LL^*$-atomic formulas.  By (1), for each  $i$, choose an $\LL$-formula 
$\psi_i(\xbar_i,\ybar_i)$ and $\mbar\in M^{\lg(\ybar_i)}$ such that
$\psi_i(\xbar_i,\mbar_i)$ equivalent on $M^*$ to $\phi_i(\xbar_i)$. Let $\mbar$ be any finite tuple containing every $\mbar_i$.  It follows by induction on the complexity of $\LL^*$-formulas
that every $\LL^*$-formula is equivalent on $M^*$ to an $\LL_{\mbar}$-formula.
\end{proof}

\begin{definition} \label{def:associated} 
Suppose $\LL_1$ and $\LL_2$ are any two languages.   An $\LL_1$-structure $M_1$ and an $\LL_2$-structure $M_2$ with the same universes are {\em associated}
if there is an $\LL_1\cup \LL_2$-structure $M^*$ (also with the same universe) that is a q.f.\ expansion by definitions of both $M_1$ and $M_2$.
\end{definition}

The most obvious example of associated structures is that any expansion $M'$ of a structure $M$ by naming constants
is associated to $M$.  The following results are immediate from Fact~\ref{fact: expansion}.

\begin{fact}  \label{fact:  fact2}  Suppose $M_1$, $M_2$ are associated structures in languages $\LL_1, \LL_2$, respectively.  Then, letting $M$ denote their common universe:
\begin{enumerate}
\item  Every (q.f.) $\LL_1(M)$-formula is equivalent to a (q.f.) $\LL_2(M)$-formula; and
\item  For $\ell=1,2$, if $N_\ell\succeq M_\ell$ then there is a unique $N_{3-\ell}\succeq M_{3-\ell}$ with universe $N_\ell$ such that $N_1$ is associated to $N_2$.
\end{enumerate}
\end{fact}

\begin{definition}  Suppose $M$ is a mutually algebraic $\LL$-structure.  A set $\A$ of quantifier-free, mutually algebraic $\LL_M$-formulas is {\em acceptable}
if every $R\in\LL$ is equivalent to a boolean combination of formulas from $\A$.  If $\A$ is an acceptable set of formulas, let
$D(\A):=\bigcup\{\abar:\phi(\xbar,\abar)\in\A\}$ denote the parameters from $M$ used in $\A$.
\end{definition}

Note that by Theorem~\ref{thm:ma char}(2), every $M$ admits an acceptable set $\A$.  Moreover, if $\LL$ is finite, then $\A$ and hence $D(\A)$ can be chosen to be finite as well.

\begin{definition} \label{def:expand} Suppose $M$ is a mutually algebraic $\LL$-structure and $\A$ is an acceptable set of $\LL_M$-formulas.
Let $$\LL_{\A}:=\set{\hbox{constant symbols of $\LL$}}\cup\set{c_d:d\in D(\A)}\cup\set{R_\phi(\xbar):\phi(\xbar,\abar)\in\A}$$
and let $M_{\A}$ be the $\LL_{\A}$-structure with universe $M$, $c^{M_{\A}}=c^M$ for every  $c\in \LL$, $c_d^{M_{\A}}=d$ for every $d\in D(\A)$,
and $R_\phi^{M_{\A}}:=\{\bbar\in M^{\lg(\xbar)}:M\models\phi(\bbar,\abar)\}$.
\end{definition}

\begin{lemma}  \label{lem:basic}  
Let $M$ be any mutually algebraic $\LL$-structure, $\A$ an acceptable set of $\LL_M$-formulas, and let $M_{\A}$ be the $\LL_{\A}$-structure constructed as above.
Then:
\begin{enumerate}
\item  $M_\A$ is MA-presented.
\item  $M$ and $M_{\A}$ are associated.
\item  For every $Y\subseteq M$ and every subset $X\subseteq M^k$,
\begin{enumerate}
\item  $X$ is $Y$-definable in $M_{\A}$ if and only if $X$ is $(D(\A)\cup Y)$-definable in $M$.
\item  $\acl_{M_{\A}}(Y)=\acl_{M}(D(\A)\cup Y)$.
\end{enumerate}
\end{enumerate}
\end{lemma}

\begin{proof}  The verifications of (1) and (2) are just unpacking definitions.  It is also clear that (3b) follows immediately from (3a), which itself follows from induction on the complexity of formulas, using that the constants of $\LL$ are present in $\LL_\AA$ and interpreted correctly to handle the case of atomic formulas.
\end{proof}

Now, $\acl_{M_{\A}}(\emptyset)$ can vary widely among associated MA-presentations  of $M$ (e.g.,  the set of `extra parameters' $D(\A)$ can be increased arbitrarily),
but we will see that $\sim$-classes sufficiently far away from $\acl_{M_{\A}}(\emptyset)$ are invariant under our choice of MA-presentation.
We begin with the following lemma.  In its proof, we use some basic facts from stability theory.  We say two tuples {\em $\mbar$ and $\mbar'$ have the same strong type},
$\stp(\mbar)=\stp(\mbar')$ if and only if $E(\mbar,\mbar')$ for every 0-definable equivalence relation with finitely many classes.  It is easily seen that having the same strong type implies having the same type, but the bonus is that if each of $\mbar,\mbar'$ are independent from $c$, i.e., neither tuple forks over $c$, then $\tp(\mbar,c)=\tp(\mbar',c)$.

\begin{lemma}  \label{lem:unique}  
Suppose $M_1$ and $M_2$ are both associated MA-presented $\LL_1$, $\LL_2$-structures respectively with common mutual expansion $M^*$, and
suppose $a\not\in\acl_{M^*}(\emptyset)$.  Then  $[a]_{\sim_{M_1}}=[a]_{\sim_{M_2}}$. 
\end{lemma}

\begin{proof}  By passing to  sufficiently saturated elementary extensions, we may assume that $M_1$, $M_2$ and $M^*$ are all $(|\LL_1|+|\LL_2|)^+$-saturated.
By symmetry, it suffices to prove that if $a\sim_{M_1} b$, then $a\sim_{M_2} b$.  For this, choose a finite tuple $\cbar$ with $a,b\in\cbar$ and a quantifier-free, mutually
algebraic $\LL_1$-formula $\phi(\xbar)\in\tp(\cbar)$.   Since $a\not\in\acl_{M^*}(\emptyset)$, the mutual algebraicity of $\phi(\xbar)$ implies that $\cbar\cap\acl_{M^*}(\emptyset)=\emptyset$.  

   As $M_2$ is associated to $M_1$, choose a quantifier-free $\LL_2(M)$-formula $\theta(\xbar,\mbar)$ equivalent to $\phi(\xbar)$ on $M$.  
   Thus, $$M^*\models\forall\xbar (\phi(\xbar)\leftrightarrow \theta(\xbar,\mbar))$$
   As $M^*$ is sufficiently saturated, choose $\mbar'$ from $M^*$ so that ${\rm stp}_{M^*}(\mbar')={\rm stp}_{M^*}(\mbar)$ and $\mbar'$ independent from $\cbar$ over $\emptyset$
   (in $M^*$).
   As $\tp_{M^*}(\cbar,\mbar')=\tp_{M^*}(\cbar,\mbar)$, we have $M^*\models\theta(\cbar,\mbar')$, and by the independence we have $\cbar\cap\acl_{M^*}(\mbar')=\emptyset$.
   But now, since $\theta(\xbar,\mbar')$ is a quantifier-free $\LL_2(M)$-formula, 
$M_2\models\theta(\cbar,\mbar')$ and since $M_2$ is a reduct of $M^*$, $\cbar\cap\acl_{M_2}(\mbar')=\emptyset$.
Thus, by Lemma~\ref{lem:support}(3) there are  $\zbar\supseteq\xbar$, $\dbar\supseteq\cbar$ with
 $\lg(\dbar)=\lg(\zbar)$, and a linked $\LL_2$-atomic conjunction $\delta(\zbar)\in\tp(\dbar)$.  As $a,b\in\dbar$, $\dbar$ witnesses that $a\sim_{M_2} b$.
\end{proof}  

It is tempting to weaken the hypothesis of Lemma~\ref{lem:unique} to simply `$a\not\in\acl_{M_1}(\emptyset)\cup\acl_{M_2}(\emptyset)$' but the following example shows this is not sufficient.

\begin{example}  
Let $\LL_1$ contain a single binary relation $R_1$ and let $M_1$ be an $\LL_1$-structure containing infinitely many directed $\Z$-chains, $\Z^+$-chains, and $\Z^-$-chains. In some
$\Z$-chain of $M_1$, pick a point $a$ and its successor $b$, although we do not  name them by constants. Let $\LL_2$
 contain a single binary relation $R_2$, and let $M_2$ be the $\LL_2$-structure with the same universe as $M_1$ and 
 $R_2^{M_2}$ interpreted as  $R_1^{M_1}\setminus\{(a,b)\}$. 
 It is easily checked that $M_1$ and $M_2$ are associated, that $\acl_{M_1}(\emptyset)=\acl_{M_2}(\emptyset)=\emptyset$, and that the $\sim$-components are the chains.

Now $M_2$ looks the same as $M_1$,  except that the  $\Z$-chain containing $m$ and $m'$  has been broken to produce a $\Z^+$-chain 
with `minimum' element $b$ and a $\Z^-$-chain with `maximum' element $a$.
Hence $a\sim_{M_1} b$, but $a\not\sim_{M_2} b$.
\end{example}

Despite this example,  if  $M_{\A_1}$ and $M_{\A_2}$ are each associated to the same mutually algebraic $M$ as in Definition~\ref{def:expand}, we can do better.

\begin{proposition} \label{prop:acc}  Suppose $M$ is a mutually algebraic $\LL$-structure with two acceptable sets $\A_1$ and $\A_2$ of $\LL_M$-formulas.
For $\ell=1,2$, let $\acl_\ell$ denote the algebraic closure relation on $M_{\A_\ell}$ and let $[a]_\ell$ denote the $\sim$-component of $a$ in $M_{\A_\ell}$.
For any $a\in M\setminus (\acl_1(\emptyset)\cup\acl_2(\emptyset))$ we have $[a]_1=[a]_2$.
\end{proposition}

\begin{proof} To ease notation, for $\ell=1,2$, write $M_\ell$ in place of $M_{\A_\ell}$ and $D_\ell$ in place of 
$D(\A_\ell)$.   As $M_1$ and $M_2$ are associated, let $M^*$ be their common q.f.\ expansion.  In light of Lemma~\ref{lem:unique}, it suffices to show 
that $a\not\in\acl_{M^*}(\emptyset)$. 
 
 As $M,M_1, M_2, M^*$ all have the same universe and the same interpretation of every constant $c\in\LL$, it follows by induction on formulas as in Lemma \ref{lem:basic} that every $X\subseteq M^k$ that is 0-definable in $M^*$ is $(D_1\cup D_2)$-definable in $M$. It then follows that $\acl_{M^*}(\emptyset)\subseteq\acl_M(D_1\cup D_2)$.  However, as $a\not\in\acl_\ell(\emptyset)$ for both $\ell=1,2$,
 it follows from Lemma~\ref{lem:basic} that $a\not\in(\acl_M(D_1)\cup\acl_M(D_2))$.  But, as $M$ is mutually algebraic, algebraic closure is trivial, hence
 $a\not\in \acl_M(D_1\cup D_2)$, hence not in $\acl_{M^*}(\emptyset)$.  Thus, $[a]_1=[a]_2$ follows from Lemma~\ref{lem:unique}.
 \end{proof}

%
%
%
%

\section{Cellularity}

We begin by stating our definition of cellularity.

\begin{definition} \label{def:cell}
A structure $M$ is {\em cellular} if, for some integers $n$ and $\<k_i:i\in[n]\>$, it admits a partition $\set{K, \set{\cbar_{i,j} | i \in [n], j \in \omega}}$ satisfying the following.
\begin{enumerate}
\item $K$ is finite and  $\lg(\cbar_{i,j})=k_i$.
\item For a fixed $i$, $\set{\cbar_{i,j} | j \in \omega}$ are pairwise isomorphic over $K$. We may thus enumerate each as $\cbar_{i,j}=(c_{i,j}^1, \dots, c_{i,j}^{k_i})$.
\item For every $i \in [n]$ and $\sigma \in S_\infty$, there is a $\sigma_i^* \in Aut(M)$ mapping each $\cbar_{i,j}$ onto $\cbar_{i, \sigma(j)}$ by sending $c_{i,j}^\ell$ to $c_{i, \sigma(j)}^\ell$ for $1 \leq \ell \leq {k_i}$, and fixing $M \bs \bigcup_{j \in \omega} \cbar_{i, j}$ pointwise.
\end{enumerate}
\end{definition}

The original definition of cellularity from \cite{Schmerl} and its rephrasing in \cite{MPW} require $n=1$. But given a partition as in Definition \ref{def:cell}, we may produce one with $n=1$ by taking $\cbar_j$ to be the concatenation of $\cbar_{i,j}$ for each $i \in [n]$, so our definition is equivalent.

Note that like $\omega$-categoricity, $\omega$-stablility, and mutual algebraicity, the cellularity of a structure $M$ is preserved under reducts.
Thus, we begin with the following Lemma that does not depend on our choice of language.

\begin{lemma}  \label{lem:three}
If $M$ is cellular, then $M$ is $\omega$-categorical, $\omega$-stable, and mutually algebraic.
\end{lemma}

\begin{proof}  Call a cellular partition $\set{K, \set{\cbar_{i,j} | i \in [n], j \in \omega}}$ of $M$
{\em indecomposable} if, for each $i\in[n]$ if we partition each $\cbar_{i,j}$ as $\dbar_{i,j}$\^{}$\ebar_{i,j}$ with both sides non-empty,
then $$\set{K, \set{\cbar_{i',j} | i' \in [n], i'\neq i,  j \in \omega}\cup\set{\dbar_{i,j}:j\in\omega}\cup \set{\ebar_{i,j}:j\in\omega}}$$
is {\em not} a cellular decomposition of $M$.

As $\lg(\cbar_{i,j})$ drops each time a cellular partition is decomposed, it follows that every cellular $M$ has an indecomposable cellular partition.
Fix an indecomposable cellular  partition  $\set{K, \set{\cbar_{i,j} | i \in [n], j \in \omega}}$ witnessing the cellularity of $M$.
Form a new  finite relational language 
$$\LL^*:=\set{c_k:k\in K}\cup\set{U_{i,\ell}:i\in[n], 1\le\ell\le k_i}\cup \set{R_i(\xbar):i\in[n]}$$ 
where each $U_{i,\ell}$ is unary and each $R_i$ is a $k_i$-ary relation symbol.  Let $M^*$ be the $\LL^*$-structure with universe $M$, and interpretations $c_k^{M^*}=k$ for each $k\in K$,  $U_{i,\ell}^{M^*}:=\set{ c_{i,j}^\ell:j\in\omega}$, and $R_i^{M^*}:=\set{\cbar_{i,j}:j\in\omega}$.  It is easily verified that the `grid-like structure' $M^*$ is $\omega$-categorical, $\omega$-stable, and mutually algebraic.
Moreover, by the automorphism condition in the definition of cellularity, it follows that every $\LL$-definable set in $M$ is $\LL^*$-definable in $M^*$, hence $M$ is a  reduct  of $M^*$.
\end{proof}

The following Lemma crucially requires $\LL^*$ to be finite relational. As an example, naming infinitely many constants  destroys  $\omega$-categoricity.

\begin{lemma}  \label{lem:slide}
Suppose $\LL,\LL'$ are both finite relational languages and the $\LL$-structure $M$ is associated to the $\LL'$-structure $M'$.  Then
\begin{enumerate}
\item  $M$ is $\omega$-categorical if and only if $M'$ is.
\item  $M$ is cellular if and only if $M'$ is.
\end{enumerate}
\end{lemma}

\begin{proof}  For both parts, as both properties are preserved under reducts, it suffices to prove this in the special case where $\LL\subset \LL'$ and 
$M'$ is a q.f.\ expansion by definitions of $M$.  As $\LL'$ is finite relational, by Fact~\ref{fact: expansion}(5), choose a finite $\mbar$ from $M'$ so that
every $\LL'$-formula is equivalent in $M'$ to an $\LL_{\mbar}$-formula.

For (1), assume $M$ is $\omega$-categorical.  Let $k:=\lg(\mbar)$.
Then, up to equivalence in $M'$, for every $n\ge 1$, the number of $\LL'$-formulas $\phi(\xbar)$ with at most $n$ free variables is at most the number of $\LL$-formulas
with at most $(n+k)$ free variables.  As the latter number is finite, $M'$ is $\omega$-categorical by Ryll-Nardzewski's theorem.

For (2), assume $M$ is cellular.  Fix a partition $\set{K, \set{\cbar_{i,j} | i \in [n], j \in \omega}}$ witnessing this.
By Lemma~\ref{lem:three}, $M$ is $\omega$-categorical, hence $M'$ is $\omega$-categorical by (1).  Let 
$$K^*:=K\cup\bigcup\set{ \cbar_{i,j}:\cbar_{i,j}\cap\mbar\neq\emptyset}$$
Note that $K^*$ is finite, and after reindexing (removing finitely many tuples) $M'$ can be partitioned as  $\set{K^*, \set{\cbar_{i,j} | i \in [n], j \in \omega}}$.
It is easily checked that this partition witnesses that $M'$ is cellular.
\end{proof}

As a warm-up to proving Theorem~\ref{thm:arbitrary}, we first characterize cellularity among  infinite, finitely MA-presented structures.

\begin{proposition}  \label{prop:finpresented}
The following are equivalent for an infinite, finitely MA-presented $\LL$-structure $M$:
\begin{enumerate}
\item  $M$ is cellular;
\item  $M$ is $\omega$-categorical;
\item  There is a uniform finite bound on $|\acl(a)|$ for $a\in M$;
\item  There is a uniform finite bound on $|[a]_\sim|$ for $a\in M$.
\end{enumerate}
\end{proposition}

We remark that this list could be extended.  For example such an $M$ is cellular if and only if the binary relation $\sim$ is $\LL$-definable.

\begin{proof}  $(1)\Rightarrow(2)$ is by Lemma~\ref{lem:three}.

$(2)\Rightarrow(3)$:   This direction holds for any    $\omega$-categorical $M$.  
As there are only finitely many distinct 1-types, and only finitely many inequivalent 2-formulas $\delta(x,y)$, there is an integer $K$ so that $M\models\exists ^{<K}x\delta(x,a)$
among all $a\in M$ and all algebraic formulas $\delta(x,a)$.  From this, a uniform bound on $|\acl(a)|$ can be found.

 $(3)\Rightarrow(4)$ is immediate, since by Proposition~\ref{prop: acl}, $[a]_\sim\subseteq\acl(a)$ for every $a\in M$.

$(4)\Rightarrow(1)$:  
Choose enumerations $\cbar$ for each $[a]_\sim$-class in $M$.  As $L$ is finite and there is a uniform bound on $\lg(\cbar)$, the set of quantifier-free types
$qftp(\cbar)$ occurring is finite.  Let $\{p_i(\xbar_i):i\in[n]\}$ enumerate the quantifier-free types that occur infinitely often, and let 
$K:=\bigcup\{\cbar:qftp(\cbar)$ occurs only finitely many times$\}$.  Then $K\subseteq M$ is finite.  For each $i\in[n]$, let $\{\cbar_{i,j}:j\in\omega\}$ enumerate
the tuples $\cbar$ realizing $p_i(\xbar_i)$.  Recalling that each $\cbar_{i,j}$ is an enumeration of a $\sim$-class, Lemma~\ref{lem: component} implies that
$\set{K, \set{\cbar_{i,j} | i \in [n], j \in \omega}}$ is a cellular decomposition of $M$.
 \end{proof}

Next, we relax the assumption that a mutually algebraic $M$ is finitely MA-presented, but keep the assumption of a finite relational language.
For this, we begin with the following lemma.

\begin{lemma}  \label{lem:transferacl}
Suppose $M$ and $M'$ are associated mutually algebraic structures in finite relational languages $\LL$ and $\LL'$, respectively.
If there is a uniform finite bound on $|\acl_M(a)|$ for each $a\in M$, then there is a uniform finite bound on $|\acl_{M'}(a)|$ for each $a\in M$.
\end{lemma}

\begin{proof}  Suppose $|\acl_M(a)|\le N$ for every $a\in M$.  Then, for any finite set $A\subseteq M$, since $\acl_M(A)=\bigcup\{\acl_M(a):a\in A\}$ for every
$a\in A$, $|\acl_M(A)|\le N\cdot|A|$.  Now,  it follows from Fact~\ref{fact: expansion}(5) that there is a finite $\mbar$ from $M$ so that every $\LL'$-formula is equivalent on $M$ to
an $\LL_{\mbar}$-formula.  Fix any $a\in M$.  From the above, if $b\in\acl_{M'}(a)$, then $b\in\acl_M(a\mbar)$.   Thus, if $k=|\mbar|$, then $\acl_{M'}(a\mbar)\le N(k+1)$.
\end{proof}

\begin{theorem}  \label{thm:finiterelational}
The following are equivalent for an infinite, mutually algebraic $\LL$-structure $M$ in a finite relational language $\LL$.
\begin{enumerate}
\item  $M$ is cellular;
\item $M$ is $\omega$-categorical;
\item  There is a uniform finite bound on $|\acl(a)|$ for all $a\in M$;
\item  For every finitely MA-presented $M'$ associated to $M$, there is a uniform finite bound on $|[a]_\sim|$ for $a\in M'$;
\item  For some finitely MA-presented $M'$ associated to $M$, there is a uniform finite bound on $|[a]_\sim|$ for $a\in M'$.
\end{enumerate}
\end{theorem}

\begin{proof} 
 $(1)\Rightarrow(2)$ is given by Lemma~\ref{lem:three}. 

$(2)\Rightarrow(3)$:  As in the proof of Proposition~\ref{prop:finpresented}, (3) holds for any $\omega$-categorical $M$.

$(3)\Rightarrow(4)$:  Choose any finitely MA-presented $M'$ associated to $M$.
By (3) and Lemma~\ref{lem:transferacl}, there is a uniform bound on $|\acl(a)|$ among $a\in M$, hence there is a uniform bound on
$|[a]_\sim|$ by Proposition~\ref{prop:finpresented} applied to $M'$.

$(4)\Rightarrow(5)$:   This follows immediately from Lemma~\ref{lem:basic}, noting that a finite acceptable set $\A$ can be found.

$(5)\Rightarrow(1)$:  Given $M'$ witnessing (5), $M'$ is cellular by Proposition~\ref{prop:finpresented}, hence $M$ is cellular by Lemma~\ref{lem:slide}.
\end{proof}

Finally, we want to identify cellular structures `in the rough' where the original language might not be finite.  
We do know that on one hand, that if  $M$ cellular, it is describable in a finite relational language (see e.g., the proof of Lemma~\ref{lem:three}). 

On the other hand, extending to an infinite language (even with only relation symbols) can be troublesome.

\begin{example} \label{ex:unary} Let $\LL=\set{U_n:n\in\omega}$ and let $M$ be a countable model of the theory of `independent unary predicates.'  Then $M$ is mutually algebraic,
and in fact, MA-presented in this language.  Additionally, $[a]_\sim=\{a\}=\acl(a)$ for all $a\in M$, hence there are uniform, finite bounds (namely one) for the sizes of these sets.
Despite that, $M$ is not cellular.
\end{example}

\begin{theorem} \label{thm:arbitrary} Suppose $M$ is a countable structure in a countable language.  Then $M$ is cellular if and only if $M$ is mutually algebraic and $\omega$-categorical.
\end{theorem}

\begin{proof}  That cellularity implies mutual algebraicity and $\omega$-categoricity is Lemma~\ref{lem:three}.  Conversely, assume $M$ is mutually algebraic and $\omega$-categorical.
We reduce to the case of a finite relational language, and so finish by Theorem~\ref{thm:finiterelational}.  Recall Lachlan's theorem, \cite[Proposition 1.6]{Lach0} that a superstable, $\omega$-categorical theory in countable language is $\omega$-stable.

 Thus, as $M$ is mutually algebraic  implies $Th(M)$ is weakly minimal and superstable, our $\omega$-categorical $M$ is $\omega$-stable as well.  Thus, by e.g., \cite[{Ch. 3, Lemmas 1.7, 3.9}]{Pill}, $Th(M)$ is
 interdefinable with some reduct $M_0$ in  a finite  relational language.
It follows that $M_0$ is cellular by Theorem~\ref{thm:finiterelational}.  But, as $M$ is interdefinable with $M_0$, $M$ is cellular as well.
\end{proof}

From these results, we obtain a characterization of cellularity that does not mention mutual algebraicity.

\begin{definition}
A countable structure $M$ is {\em monadically $\omega$-categorical} if every expansion of $M$ by (finitely many) unary predicates is $\omega$-categorical.
\end{definition}

\begin{theorem} \label{thm:mon cat}
  A countable structure $M$ is monadically $\omega$-categorical if and only if $M$ is cellular.
\end{theorem}
\begin{proof}
It is easy to see that if $M$ is cellular, then any expansion by finitely many unary predicates still admits a partition as in Definition \ref{def:cell}, although the value of $n$ for the partition might increase.

Now suppose $M$ is monadically $\omega$-categorical.  As $M$ itself is $\omega$-categorical, by Theorem~\ref{thm:arbitrary} it suffices to show that $M$ is mutually algebraic.
Assume this is not the case.   Then, by combining \cite[{Theorem 3.3(3)}]{MA} with \cite[{Theorem 3.2(4)}]{Tri}, we conclude there is an expansion $M^*$ of $M$ by finitely many unary predicates such that there are $M^*$-definable $D\subseteq (M^*)^1$ and $E\subseteq D^2$, where $E$ is an equivalence relation on $D$ with infinitely many infinite classes $\{C_n:n\in\omega\}$.  But then, adding a single, new unary predicate $U$ containing exactly $n$ elements of each $C_n$, we get a monadic expansion of $M$ with infinitely many 1-types, contradicting $M$ being monadically $\omega$-categorical.
%
\end{proof}

We remark that the restriction of adding only finitely many unary predicates is necessary, see e.g., Example~\ref{ex:unary} -- If $\LL$ is just equality and $M$ is an infinite $\LL$-structure, then $M$ is cellular, but there is an expansion by $\omega$ unary predicates that is not cellular.

Our final result gives another characterization of cellularity among mutually algebraic structures and allows us to produce an extension of a mutually algebraic non-cellular model analogous to an extension of a non-mutually algebraic model by infinitely many infinite arrays all realizing the same quantifier-free type. Thus we may reproduce the ``many models'' behavior of the non-mutually algebraic case in the mutually algebraic non-cellular case.

\begin{theorem} \label{thm:ma noncell}
Let $\LL$ be finite relational, and suppose $M$ is mutually algebraic but non-cellular countable $\LL$-structure. Then there is some $M^* \succ M$ such that $M^*$ contains infinitely many new infinite MA-connected components, pairwise isomorphic over $M$. Furthermore, we may take the universe of $M^*$ to be the universe of $M$ together with these new components.
\end{theorem}

\begin{proof}
As $\LL$ is finite relational, by Lemma~\ref{lem:basic} there is an expansion $M'$ of $M$ by finitely many constants with $M'$ MA-presented.  As the $\LL$-reduct of any elementary extension of $M'$ is an $\LL$-elementary extension of $M$, we may assume $M$ itself is MA-presented.
By Proposition \ref{prop:finpresented}, there is no finite bound on the size of MA-connected components in $M$.

\begin{claim}
We may find an elementary extension of $M$ adding a single new infinite MA-component.
\end{claim}
\begin{claimproof}
We proceed by compactness. Let $r$ be the maximum arity of the language $\LL$ and 
expand $\LL_M$ by new constants $\set{c_{i,j} | i \in \omega, j \in [r]}$.
As notation,  let $\cbar_i = (c_{i,1}, \dots, c_{i,r})$. Consider the following  theory:

\begin{enumerate}
\item The elementary diagram of $M$.
\item For every $i \in \omega$, $c_{i,j}$ is not equal to any $m \in M$.
\item For every $i \in \omega$, some $R \in \LL$ holds on some initial subtuple of $\cbar_i$. We now let $\cbar'_i$ denote the maximal initial subtuple such that some $R \in \LL$ holds on $\cbar'_i$.
\item For every $i \in \omega$, $\cbar'_i \cap \cbar'_{i+1} \neq \emptyset$.
\item For every $i \neq j \in \omega$, $\cbar'_i \neq \cbar'_j$.
\end{enumerate}

If this theory were satisfiable, we would get an elementary extension of $M$, and $[c_{0,0}]_\sim$ would be the desired infinite MA-connected component. So consider a finite subset $S$ of the sentences, and let $F \subset M$ be the finite set of realizations of $\LL_M$-constants mentioned in $S$. Let $n \in \omega$ be such that $n-1$ is the largest first index of any $c_{i,j}$ occurring in $S$. We must interpret $(\cbar'_0, \dots, \cbar'_{n-1})$ so that it satisfies a linked $\LL$-atomic conjunction, avoids $F$, and the subtuples are distinct. Given $x, y \in M$, let $d(x, y)$ be the minimum number of conjuncts of any $\LL$-atomic conjunction satisfied by any tuple containing $x$ and $y$ (and $d(x, y) = \infty$ if they are in different MA-connected components). Then it suffices to find some $x \in M$ such that $|[x]_\sim| > nr$ and $d(x, f) > n$ for every $f \in F$, since we may then find a tuple containing $x$ and avoiding $F$ that satisfies a linked $\LL$-atomic conjunction with $n$ conjuncts (and with the subtuples for each conjunct distinct), and may take $\cbar'_i$ to be the subtuple satisfying the $i^{th}$ conjunct.

Since $M$ is MA-presented, there is some $K \in \omega$ such that for every $m \in M$, $|\set{x \in M | d(m, x) \leq k}| \leq K^k$ for every $k \in \omega$. As there is no finite bound on the size of MA-connected components in $M$, there is some MA-connected component of size greater than $|F|\cdot K^n$, and so we may take the desired $x$ in that component. We may also need to interpret constants appearing in $\bigcup_i \set{\cbar_i \bs \cbar'_i}$, but we may choose all such points to lie in a different MA-connected component.
\end{claimproof}

With the claim in hand, we work inside a sufficiently saturated elementary extension of $M$. We may take an infinite set $\set{c_i| i \in \omega}$ of realizations of $\tp(c_{0,0}/M)$, pairwise independent over $M$. Then $\set{[c_i]_\sim| i \in \omega}$ are the desired infinite MA-connected components, pairwise isomorphic over $M$. Furthermore, by \cite[{Proposition 4.2}]{MA}, $M \prec M \cup \set{[c_i]_\sim|i \in \omega}$.
\end{proof}


\section{Declarations}
 \subsection{Conflict of interest}
The authors declare that they have no conflict of interest.

 \subsection{Availability of data and code}
 
 The manuscript has no associated data or code.

\bibliography{Bib.bib}
\bibliographystyle{spmpsci}      

%
%

\end{document}